\documentclass[12pt]{amsart}
\usepackage{} 
\usepackage{amsmath,amssymb,amsthm}
\usepackage{amscd}
\usepackage{mathrsfs}
\usepackage[usenames,dvipsnames]{color}
\usepackage{pstricks}
\usepackage{graphicx}
\usepackage[all]{xy}
\usepackage{extarrows}
\usepackage{tabularx}
\usepackage{slashed}
\usepackage{tikz}
\usepackage{xparse}
\usepackage{enumitem}
\usepackage[normalem]{ulem}

\usepackage{ifpdf}
\ifpdf
  \usepackage[colorlinks,final,backref=page,hyperindex]{hyperref}
\else
  \usepackage[colorlinks,final,backref=page,hyperindex,hypertex]{hyperref}
\fi

\usepackage{xspace}

\usetikzlibrary{arrows,matrix}
\usetikzlibrary{decorations.pathmorphing}

\newtheorem{theorem}{Theorem}[section]
\newtheorem{lem}[theorem]{Lemma}

\newtheorem{prop}[theorem]{Proposition}

\theoremstyle{definition}
\newtheorem{defn}[theorem]{Definition}

\newtheorem{exam}[theorem]{Example}

\newcommand{\nc}{\newcommand}
\newcommand{\delete}[1]{}

\nc{\mlabel}[1]{\label{#1}}  
\nc{\mcite}[1]{\cite{#1}}  
\nc{\mref}[1]{\ref{#1}}  
\nc{\meqref}[1]{\eqref{#1}}  
\nc{\mbibitem}[1]{\bibitem{#1}} 

\delete{
\nc{\mlabel}[1]{\label{#1}  
{\hfill \hspace{1cm}{\small\tt{{\ }\hfill(#1)}}}}
\nc{\mcite}[1]{\cite{#1}{\small{\tt{{\ }(#1)}}}}  
\nc{\mref}[1]{\ref{#1}{{\tt{{\ }(#1)}}}}  
\nc{\meqref}[1]{\eqref{#1}{{\tt{{\ }(#1)}}}}  
\nc{\mbibitem}[1]{\bibitem[\bf #1]{#1}} 
}

\nc{\mOmega}{X}

\nc{\groupoid}{magma\xspace}
\nc{\groupoids}{magmas\xspace}

\nc{\mtop}{\top\hspace{-1mm}}
\nc{\mim}{\text{im}}
\nc{\tforall}{\text{ for all }}
\nc{\bfcf}{{\calc}}

\nc{\PP}{\mathbb{P}}
\newcommand{\Q}{\mathbb{Q}}

\newcommand {\calc}{{\mathcal {C}}}

\newcommand {\call}{{\mathcal {L}}}

\newcommand {\calp}{{\mathcal {P}}}

\newcommand {\calq}{{\mathcal {Q}}}
\newcommand {\cals}{{\mathcal {S}}}

\nc{\Id}{\mathrm{Id}}
\nc{\ot}{\otimes}
\nc{\bt}{\boxtimes}
\nc{\id}{\mathrm{Id}}
\nc{\Hom}{\mathrm{Hom}}
\nc{\im}{{\mathfrak Im}}

\nc{\mge}{_{bu}\!\!\!\!{}}

\nc{\vep}{\varepsilon}

\nc{\syd}[1]{}

\nc{\prt}{P-}
\nc{\Prt}{P-}

\nc{\orth}{orthogonal\xspace}
\nc{\tloc}{L-}

\nc{\gnt}{{G^{_\top n}}}

\nc{\BA}{{\mathbb A}} \nc{\CC}{{\mathbb C}} \nc{\DD}{{\mathbb D}} \nc{\EE}{{\mathbb E}} \nc{\FF}{{\mathbb F}} \nc{\GG}{{\mathbb G}} \nc{\HH}{{\mathbb H}} \nc{\NN}{{\mathbb N}} \nc{\KK}{{\mathbb K}}
\nc{\QQ}{{\mathbb Q}} \nc{\RR}{{\mathbb R}} \nc{\TT}{{\mathbb T}} \nc{\VV}{{\mathbb V}} \nc{\ZZ}{{\mathbb Z}}

 \nc{\bfk}{{\bf k}}
\nc{\ID}{\mathfrak{I}} \nc{\lbar}[1]{\overline{#1}}
\nc{\bre}{{\rm b}} \nc{\sd}{\cals} \nc{\rb}{\rm RB}
\nc{\A}{\rm angularly decorated\xspace} \nc{\LL}{\rm L}
\nc{\w}{\rm wid} \nc{\arro}[1]{#1}
\nc{\ver}{\rm ver}
\nc{\FL}{F_{\mathrm L}}
\nc{\FNA}{\FN(A)} \nc{\NA}{N_{A}}
\nc{\dr}{\diamond_r}
\nc{\shar}{{\mbox{\cyrs X}}_r} 
\nc{\dt}{\Delta_T}
\nc{\da}{\Delta_A}
\nc{\vt}{\vep_T }
\nc{\bul}{\bullet}
\nc{\free}[1]{\bar{#1}}
\nc{\lt}{{}^\top\!U}
\nc{\rt}{U^\top}
\nc{\lts}{{}^\top\!}
\nc{\weak}{strong\xspace}
\nc{\strong}{strong\xspace}
\nc{\fine}{rigid\xspace}

\nc{\full}{smooth\xspace}
\nc{\ful}{\mathrm{sm}}
\nc{\fuly}{smoothness\xspace}
\nc{\sfull}{{\rm sm}}
\nc{\reduced}{reduced\xspace}

\nc{\sx}{\sigma}
\nc{\tx}{{\tau}}
\nc{\sq}{{s_Q}\xspace}
\nc{\tq}{{t_Q}\xspace}
\nc{\ol}{\overline}
\nc{\mreg}{closed\xspace}
\nc{\sreg}{{\rm cl}}
\nc{\mregy}{closeness\xspace}
\nc{\reg}{{\text{cl}}}

\begin{document}

\title[Quivers, paths and locality]{Quivers and path semigroups characterized by locality conditions}

\author{Li Guo}
\address{
	Department of Mathematics and Computer Science,
	Rutgers University,
	Newark, NJ 07102, United States}
\email{liguo@rutgers.edu}

\author{Shanghua Zheng}
\address{Department of Mathematics, Jiangxi Normal University, Nanchang, Jiangxi 330022, China}
\email{zhengsh@jxnu.edu.cn}

\date{\today}

\begin{abstract}
The notion of locality semigroups was recently introduced with motivation from locality in convex geometry and quantum field theory.
We show that there is a natural correspondence between locality sets and quivers which leads to a concrete class of locality semigroups given by the paths of quivers. Further these path semigroups from paths are precisely the free objects in the category of locality semigroups with a \fine condition. This characterization gives a universal property of path algebras and at the same time a combinatorial realization of free \fine locality semigroups.
\end{abstract}

\subjclass[2010]{20M05, 
	 16G20, 	
	 18B40, 	
	 05C38, 	
	 08A55		
	 08B20		
 }

\keywords{locality, path algebra, free objects, quiver, partial semigroup}

\maketitle

\tableofcontents

\setcounter{section}{0}

\allowdisplaybreaks

\section{Introduction}
The notion of locality is widely  employed in various branches of mathematics, such as local algebras~\mcite{JS} and local operators in functional analysis~\mcite{BAE,Bat}.   Locality is also used in  computer science and physics, especially in classic and quantum field theory.  For instance, the locality principle (often called Einstein causality) is a critical factor in relativity and quantum field theory~\mcite{R,R2}.  It is well known that renormalization, which is a technique to remove the divergences in  Feynman integrals calculations,  plays an important role in quantum field theory~\mcite{Bor,BP,CK,Z} and mathematics~\mcite{GZ,GPZ3,M}.
More recently, an algebraic approach to preserve locality in the renormalization in the framework of algebraic Birkhoff factorization was proposed~\mcite{CGPZ1,CGPZ2, CGPZ3}. For this purpose, a hierarchy of locality algebraic structures are established, starting with locality semigroups and culminating to locality Hopf algebras and locality Rota-Baxter algebras.  See~\mcite{GPZ4} for a survey. 

From a classic point of view, locality semigroups can be regarded as a partial semigroup for which the restriction of multiplication is precisely prescribed by a locality relation.
On the other hand, a natural supply of locality semigroups are provided by path algebras of quivers.

As a basic notion in mathematics, quivers have been studied algebraically in terms of quiver representations and path algebras, with wide applications to invariant theory, Kac-Moody Lie algebras and quantum groups~\mcite{FH,HJ,Scho}.
The well-known Gabriel Theorem shows that a basic algebra over an algebraically closed field is isomorphic to a quotient of the path algebra of its Ext-quiver modulo an admissible ideal~\mcite{ASS,ARS}. More generally,  an Artinian algebra over a perfect field is given by a quotient of the generalized path algebra of its natural quiver~\mcite{LL}.
Recently, combinatorial properties of the Hochschild cohomology of path algebras is studied in terms of the genus of the quiver~\mcite{LT}.
Furthermore, the Hopf algebras on path algebras,  reconstruction of path algebras and the dimensions of path algebras were studied ~\mcite{AH,HU,KO,KW}.

To better understand the importance of quivers and path algebras, it is desirable to find intrinsic and categorical properties that characterize them. One natural question in this regard is whether quivers and path algebras have certain universal properties in suitable categories. The multiplication of quivers and paths are only partially defined, suggesting that a universal property might need to be formulated in the locality setting.

This is the motivation of our study. Indeed we show that such a characterization can be achieved using the language of locality semigroups. The precision in the notion of locality semigroups allows for a more detailed study of quivers. In fact, there is a locality condition that defines the category of \fine locality semigroups for which the locality semigroups of quivers are precisely the free objects. This result gives the desired categorical characterization of path semigroups, while at the same time, gives a concrete interpretation of free \fine locality semigroups. As a byproduct, we obtain a one-to-one correspondence between the quivers from locality sets and locality sets from quivers.

Here is an outline of the paper.

First in Section~\mref{subsec:locsg}, we give correspondences between locality sets and quivers. In particular, the locality set from a quiver gives a locality semigroup structure on the paths of the quiver. We then give a characterization of the locality sets from quivers (Proposition~\mref{prop:quireg}), in terms of the notion of \mreg locality sets. Every locality set has a unique envelop by a \mreg locality set. Further, when a quiver is sent to a locality set and then back to a quiver, the path locality semigroup does not change.

In Section~\mref{ss:all}, we also give a characterization of the quivers coming from locality sets, in terms of the notion of \full quivers. Every quiver has a unique cover by a \full quiver. Then it is shown that the previous correspondence between locality sets and quivers restricts to a one-to-one correspondence between \mreg locality sets and \full quivers (Theorem~\mref{thm:qxq}).

Finally in Section~\mref{subsec:fineLSG}, we introduce a notion of \fine locality semigroups to describe the class of locality semigroups from quivers. We then show that locality semigroups from quivers are precisely the free \fine locality semigroup generated by the given locality set. Moreover, it is shown that the locality semigroup of every quiver is a free \fine locality semigroup (Theorem~\mref{thm:freeloc1}).

\section{Locality sets and quivers}
\mlabel{subsec:locsg}

We briefly recall basic notions of locality semigroups and refer the reader to ~\mcite{CGPZ1,CGPZ2,CGPZ3} for further details.

\begin{defn}
\begin{enumerate}
\item
A {\bf locality set} is a pair $(X,\top)$, where $X$ is a set and
$$\top:=\mtop_X:=X\times_\mtop X\subseteq X\times X$$
is a binary relation on $X$, called the {\bf locality relation} of the locality set.
\item Let $(X,\mtop_X)$ be a locality set. Let $X'$ be a subset of $X$ and let $\mtop_{X'}:=(X'\times X')\cap \mtop_X$. The pair $(X',\mtop_{X'})$ is called a {\bf sub-locality set} of $(X,\mtop_X)$.
\item
A locality set $(X,\top)$ together with a partially defined binary operation
$$\mu:X\times_{\top} X\to X,\quad (x,y)\mapsto \mu(x,y)\,\, \text{for all}\,\, (x,y)\in X\times_{\top} X,$$
is called a {\bf locality \groupoid}, denoted by $X:=(X,\top,\mu)$. The image $\mu(x,y)$ is abbreviated as $x\cdot y$ or $xy$ if its meaning is clear from the context.
\item
A {\bf locality semigroup} is a locality \groupoid $(S,\top,\mu)$ that, for $(a,b), (a,c), (b,c)\in \top$, we have
$(ab,c), (a,bc)\in \top$ and the {\bf locality associative law} holds:
\begin{equation}
(a b) c=a (b c).
	\mlabel{eq:locass}
\end{equation}
A locality semigroup $(S,\top,\mu)$ is abbreviated as $(S,\top)$ if the meaning of $\mu$ is understood.
\end{enumerate}
\mlabel{defn:locsg}
\end{defn}
A locality \groupoid is also called a {\bf partial groupoid}~\mcite{Ev,LE,MPS}.

\begin{defn}
\begin{enumerate}
\item
Let $(X,\mtop_X)$ and $(Y,\mtop_Y)$ be locality sets. A map $\phi:X\to Y$ is called a {\bf locality map} if it preserves the localities in the sense that
$$(\phi \times \phi)(\mtop_X)\subseteq \mtop_Y.$$
\item Let $(S_1,\mtop_{S_1},\cdot_{S_1})$ and $(S_2,\mtop_{S_2},\cdot_{S_2})$ be locality semigroups. A locality map $\phi:S_1\to S_2$ is called a {\bf locality semigroup homomorphism} if it is {\bf locality multiplicative} in the sense that
$$\phi(a\cdot_{S_1} b)=\phi (a)\cdot_{S_2}\phi(b) \quad \text{for all } (a,b)\in \mtop_{S_1}.$$
\end{enumerate}
\end{defn}

See~\mcite{CGPZ1,CGPZ2,CGPZ3} for a systemic study of locality sets, locality semigroups and the related structures, as well as as their application in renormalization of quantum field theory. We will focus on their close relationship with quivers and path algebras. We first recall the basic notions for later applications.

A {\bf quiver}  is a quadruple $ Q:=(Q_0,Q_1,s,t)$ consisting of
\begin{enumerate}
\item
a finite set $Q_0$ of {\bf vertices};
\item
a finite set $Q_1$ of {\bf arrows};
\item
a map $s:Q_1\to Q_0$, called the {\bf source function}, and a map $t:Q_1\to Q_0$,  called the {\bf target function}.
\end{enumerate}

We shall use $x,y,z,\cdots$ to denote the vertices in $Q_0$ and use $\alpha, \beta,\gamma,\cdots$ to denote the arrows in $Q_1$. For $\alpha \in Q_1$,  we call $s(\alpha)$ (resp. $t(\alpha)$) the {\bf source} (resp. {\bf target}) of $\alpha$. An arrow with a source $x$ and a target $y$ will be denoted by $\alpha:x\to y$, or simply by $x\overset{\alpha}{\to} y$. A {\bf homomorphism}  $f:(Q_0,Q_1,s,t)\to (Q'_0,Q'_1,s',t')$ between two quivers is a pair of maps $f_0:Q_0\to Q'_0$ and $f_1:Q_1\to Q'_1$ such that $s'\circ f_1=f_1\circ s$ and $t'\circ f_1=f_1\circ t$.

Let $Q:=(Q_0,Q_1,s,t)$ be a quiver.
\begin{enumerate}
\item
A {\bf path} in $Q$ is either a vertex $v\in Q_0$, called an {\bf empty path} or a {\bf trivial path} and denoted by $e_v$, or a sequence $p:=\alpha_1\alpha_2\cdots\alpha_k$ of arrows $\alpha_i\in Q_1$ for $1\leq i\leq k, k\geq 1$, and  $t(\alpha_i)=s(\alpha_{i+1})$ for $1\leq i\leq k-1$.
\item
For a nonempty path $p=\alpha_1\alpha_2\cdots\alpha_k$ with $\alpha_i\in Q_1$, $i=1,\cdots,k$. We call $s(p):=s(\alpha_1)$ the {\bf source} of $p$ and  $t(p):=t(\alpha_k)$ the {\bf target} of $p$.
\end{enumerate}
When there is no danger of confusion, we also denote a quiver by $Q=(Q_0,Q_1)$.

We denote by $\calp_Q$ the set of \emph{nonempty} paths in a quiver $Q$.
The {\bf reduced path algebra} of a quiver $Q$ is the free module $\bfk\, \calp$ equipped with the multiplication $p\cdot q=pq$ if $t(p)=s(q)$ and $p\cdot q=0$ otherwise. Note the difference between an empty path $e_v$ and the loop given by $(v,v)\in Q_1$. $e_v$ is an idempotent in the full path algebra, but is excluded from $\calp_Q$. Loops are allowed in $\calp_Q$.

There are natural correspondences between quivers and locality set, regarded as a relation, given by the digraph representation of a relation.

A quiver naturally defines a locality set.
\begin{defn}
	Let $Q=(Q_0,Q_1,s,t)$ be a quiver. Define
\begin{equation} \mlabel{eq:quiloc}
X_Q:=Q_1, \quad \mtop_{X_Q}:=\{(\alpha,\beta)\,|\,t(\alpha)=s(\beta),\alpha,\beta\in Q_1\}.
\end{equation}
	The resulting locality set
	$\mOmega_Q:=(\mOmega_Q,\mtop_{\mOmega_Q})$ is called the {\bf locality set of the quiver $Q$}.
	\mlabel{defn:quiloc}
\end{defn}

\begin{exam}Let $Q$ be the quiver
	$$  
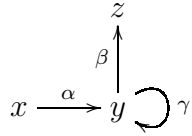
\xymatrix{& z&\\
		x \ar[r]^{\alpha} &y\ar[u]^{\beta} \ar@(ur,dr)[]^{\gamma}} $$	
	Then $X_Q=\{\alpha,\beta,\gamma\}$ and $\mtop_{X_Q}=\{(\alpha,\beta),(\alpha,\gamma),(\gamma,\gamma),(\gamma,\beta)\}$.
\mlabel{ex:quiloc}
\end{exam}

Furthermore, the paths of a quiver naturally defines a locality semigroup.
\begin{prop}
	Let $Q$ be a quiver and
	let $\calp:=\calp_Q$ be the set of nonempty paths in $Q$.  Denote
	$$\mtop_\calp:=\calp \times_\top \calp:=\{(p,q)\in \calp\times \calp\,|\, t(p)=s(q)\}.$$
	Define a partial binary operation
	$$\mu_\calp:\calp\times_\top\calp\to \calp, \quad (p,q)\mapsto pq, \quad \text{ for all } (p,q)\in \mtop_\calp.$$
	Then $\calp:=\calp_Q:=(\calp,\mtop_\calp,\mu_\calp)$  is a locality semigroup, called the {\bf path locality semigroup} of $Q$.
	\mlabel{prop:pathsg}
\end{prop}

\begin{proof}
	Let $p, q, r\in \calp$ with $(p,q), (p,r), (q,r)\in \mtop_\calp$. Then by the definition of  paths, $pq, qr$ are also paths in $\calp$ and $(pq,r), (p,qr)\in \mtop_\calp$. Further $(pq)r=p(qr)$ since the multiplication of paths is associative.
\end{proof}

In the other direction, given a locality set $\mOmega:=(\mOmega,\top)$, there is a natural quiver given by the diagraph $G=G_X$ of the relation $\top$, that is,
$G=(G_0,G_1,s_G,t_G)$ with the sets
$G_0=\mOmega, G_1=\top$ and the maps $s_G:\top \to X, (x,y)\mapsto x, t_G:\top \to X, (x,y)\mapsto y$.

We will define another quiver $Q_\mOmega:=(Q_0,Q_1,\sq,\tq)$ from $(\mOmega,\top)$ as follows, that better preserves the multiplicative structures. First define $Q_1:=\mOmega$. Next let
$\mOmega_s$ and $\mOmega_t$
be disjoint sets in bijection with $X$ by bijections
$$\sx:\mOmega\to \mOmega_s,\,\omega\mapsto \omega_s, \quad \tx:\mOmega\to \mOmega_t,\,\omega\mapsto \omega_t, \quad \omega\in X.$$
Define a relation $\sim$ on the disjoint union $\mOmega_s\sqcup \mOmega_t$ by
\begin{equation}
\alpha_t\sim \beta_s\quad \text{if}\quad (\alpha,\beta)\in\top.
\mlabel{eq:defnsim}
\end{equation}
Let $\equiv$ denote the equivalence relation on $\mOmega_s\sqcup \mOmega_t$ spanned by $\sim$. Then define $Q_0$ to be the quotient of the equivalence classes modulo $\equiv$:
$$ Q_0:=(\mOmega_s\sqcup \mOmega_t)/\!\!\equiv.$$
Let $\ol{\omega}$ denote the equivalence class of $\omega\in \mOmega_s\sqcup \mOmega_t$ modulo this equivalence relation and define
\begin{equation}
 \sq: Q_1\to Q_0, \alpha\mapsto \ol{\alpha_s}, \quad
\tq: Q_1\to Q_0, \alpha\mapsto \ol{\alpha_t}, \quad \alpha\in Q_1.
\mlabel{eq:locst}
\end{equation}

\begin{defn}
For a given locality set $X$, with the above data, the quadruple  $Q:=Q_\mOmega=(Q_0,Q_1,\sq,\tq)$ is called  the {\bf quiver of $\mOmega$}.
We then call the path locality semigroup $\calp=\calp_{Q_\mOmega}$ from this quiver defined in Proposition~\mref{prop:pathsg} the {\bf path locality semigroup of the locality set $\mOmega$}.
\mlabel{defn:locqui}
\end{defn}

We will show that path locality semigroup from a locality set satisfies a universal property in Theorem~\mref{thm:freeloc1}. In fact every path locality semigroup is isomorphic to a path locality semigroup from a locality set (Proposition~\mref{pp:pathsame}) and hence has the above universal property.

We first give an example to demonstrate the construction of the quiver of a locality set.
\begin{exam}
Let $X=\{\alpha,\beta\}$ and let $\mtop_X=\{(\alpha,\beta)\}$. Then $Q_1=X$. Let $\mOmega_s:=\{\alpha_s,\beta_s\}$ and $\mOmega_t:=\{\alpha_t,\beta_t\}$, with
the bijections
$$\sx: \mOmega\to \mOmega_s, \omega\mapsto \omega_s; \quad
\tx: \mOmega\to \mOmega_t, \omega\mapsto \omega_t, \quad \omega\in X.$$
Then  by $\mtop_X=\{(\alpha,\beta)\}$, we have $\sim=\{(\alpha_t,\beta_s)\}$ which spans the equivalence relation
$$\equiv \,=\{(\alpha_t,\beta_s), (\beta_s, \alpha_t), (\alpha_t,\alpha_t),(\beta_s,\beta_s),(\alpha_s,\alpha_s),(\beta_t,\beta_t)\}.$$
Thus
$$Q_0:=( \mOmega_s\sqcup  \mOmega_t)/\!\!\equiv\, =\{\ol{\alpha_t}(=\ol{\beta_s}),\ol{\alpha_s},\ol{\beta_t}\}.$$
Further the maps $s_Q, t_Q:Q_1\to Q_0$ in Eq.~\meqref{eq:locst} are defined by
$$s_Q(\alpha)=\overline{\alpha_s}, s_Q(\beta)=\overline{\beta_s}=\overline{\alpha_t},
t_Q(\alpha)=\overline{\alpha_t}, t_Q(\beta)=\overline{\beta_t}.$$
Then the quiver $Q_\mOmega=\{Q_0,Q_1,\sq,\tq\}$ of $X$ is
	$$  \xymatrix{& \overline{\beta_t}&\\
	\overline{\alpha_s} \ar[r]^(.4){\alpha} &\overline{\alpha_t}=\overline{\beta_s}\ar[u]^{\beta}}  $$	
\mlabel{exam:ab}
\end{exam}

A natural question is whether every locality relation can be obtained from a quiver in this way. In general this implication does not hold. To deal with the situation, we introduce a notion.

\begin{defn}
A locality set $X:=(X,\mtop_X)$ is called {\bf \mreg} if for all $ \alpha,\beta,\alpha_1,\beta_1\in X$, from $
(\alpha,\beta_1),(\alpha_1,\beta_1),(\alpha_1,\beta)\in \mtop_X$, we obtain $(\alpha,\beta)\in \mtop_X$.
\mlabel{defn:regls}
\end{defn}

The definition suggests the following first example.
\begin{exam}
Let $X=\{\alpha,\beta_1,\alpha_1,\beta)$ with
$\top_X=\{(\alpha,\beta_1), (\alpha_1,\beta), (\alpha_1,\beta)\}$. This is not \mreg. If $(\alpha,\beta)$ is added to $\top_X$, then the new locality set is \mreg.
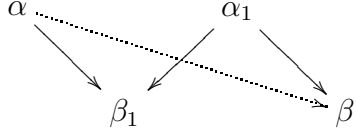
$$ \xymatrix{\alpha \ar[dr] \ar@{.>}[drrr] & & \alpha_1 \ar[dl] \ar[dr] & \\
	& \beta_1 && \beta}
$$
\mlabel{ex:regset1}
\end{exam}
We give another example.
\begin{exam}
	Let $X=\{\alpha, \beta_1,\alpha_1\}$ and let $\mtop=\{(\alpha,\beta_1),(\alpha_1,\beta_1),(\alpha_1,\alpha_1)\}$.
This can be regarded as the degenerated case of Example~\mref{ex:regset1} when $\beta=\alpha_1$. Then $(X,\top)$ is not \mreg locality set. $(\alpha,\alpha_1)\}$ is added to $\top$, then $(X,\mtop)$ is \mreg.
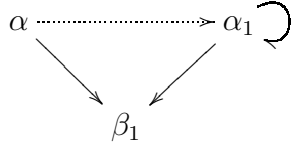
$$ \xymatrix{\alpha \ar[dr] \ar@{.>}[rr]  & & \alpha_1\ar@(ur,dr) \ar[dl]  \\
	& \beta_1 & }
$$
\end{exam}

The \mregy condition is equivalent to the follow stronger statement.
\begin{lem}
A locality set $(X,\top_X)$ is \mreg if and only if,
for all $ n\in\NN^+$, and $ \alpha, \alpha_i, \beta, \beta_i,\in X,i=1,\cdots,n$,
when $(\alpha,\beta_1),(\alpha_1,\beta_1),(\alpha_1,\beta_2), (\alpha_2,\beta_2),\cdots,(\alpha_n,\beta_n),(\alpha_n,\beta)$ are in $\mtop_X$, then $ (\alpha,\beta)\in \mtop_X$.
\mlabel{lem:regseq}
\end{lem}

\begin{proof}
We prove by induction on the integer $n\geq 1$ in the statement of the lemma, with $n=1$ being the definition. For any $k\geq 1$, assume that the lemma has been proved for $k=n$. Then for a string
$$(\alpha,\beta_1),(\alpha_1,\beta_1),(\alpha_1,\beta_2), (\alpha_2,\beta_2),\cdots,(\alpha_k,\beta_k),(\alpha_k,\beta_{k+1}), (\alpha_{k+1},\beta_{k+1}),(\alpha_{k+1},\beta)$$
in $\mtop_X$,
by definition, from the first three pairs in the string we obtain $(\alpha,\beta_2)\in \top_X$. Then replacing this pair with the first three pairs reduces to the case of $n=k$.
Then $ (\alpha,\beta)\in \mtop_X$, completing the induction.	
\end{proof}

Then we have
\begin{lem}
Let $X:=(X,\mtop_X)$ be a \mreg locality set. Let $Q:=Q_\mOmega=(Q_0,Q_1,\sq,\tq)$  be  the  quiver of $X$. Then for all $\alpha,\beta \in X(=Q_1)$,
$(\alpha,\beta)\in \mtop_X$ if and only if $\tq(\alpha)=\sq(\beta)$.
\mlabel{lem:locequiv}
\end{lem}

\begin{proof}
($\Longrightarrow$) If $(\alpha,\beta)\in\top$, then $\alpha_t\sim \beta_s$ by Eq.~(\mref{eq:defnsim}). Since  $\equiv$ is spanned by $\sim$, we have $(\alpha_t,\beta_s)\in\,\, \equiv$. Then $\ol{\alpha_t}=\ol{\beta_s}$, and so $\tq(\alpha)=\ol{\alpha_t}=\ol {\beta_s}=\sq(\beta)$.

\smallskip

\noindent
($\Longleftarrow$) Let $\alpha,\beta \in Q_1=X$. Suppose that  $\tq(\alpha)=\sq(\beta)$. Then $\ol{\alpha_t}=\ol{\beta_s}$, and so $\alpha_t\equiv \beta_s$.

Note that $\equiv$ is the equivalence relation spanned by $\sim$, and $\sim$ only relates an element of $X_s$ and $X_t$. Denote the inverse relation $\sim'$ of $\sim$ by $y_t\sim' x_s$ if $x_s\sim y_t$ for $x, y\in X$.
Then $\alpha_t\equiv \beta_s$ means that $\alpha$ and $\beta$ are connected by an alternating string of $\sim$ and $\sim'$. More precisely, there exist $n\geq 1$ and  $\alpha_i, \beta_i, 1\leq i\leq n,$ such that,

$$ \alpha_t\sim \beta_{1,s} \sim' \alpha_{1,t} \sim \beta_{2,s}\sim'\alpha_{2,t} \cdots \sim \beta_{n,s} \sim' \alpha_{n,t} \sim \beta_s.$$
That is, the pairs
$$ (\alpha,\beta_1),(\alpha_1,\beta_1), (\alpha_1,\beta_2), (\alpha_2,\beta_2), \cdots, (\alpha_n,\beta_n), (\alpha_n,\beta)$$
are in $\top_X$.
Since $(X,\top_X)$ is \mreg, Lemma~\mref{lem:regseq} yields $(\alpha,\beta)\in \mtop_X.$ This completes the proof.
\end{proof}

\begin{prop}
For any quiver $Q:=(Q_0,Q_1,s,t)$, the locality set $X_Q:=(X_Q,\mtop_{X_Q})$  of $Q$ as defined in Definition~\mref{defn:quiloc} is a \mreg locality set.
\mlabel{prop:quireg}
\end{prop}
\begin{proof}
By Definition~\mref{defn:quiloc}, we obtain
 $$\mtop_{X_Q}=\{(\alpha,\beta)\,|\,t(\alpha)=s(\beta)\}\subseteq X_Q\times X_Q.$$
If $(\alpha,\beta_1),(\alpha_1,\beta_1),(\alpha_1,\beta_2)\cdots,(\alpha_\ell,\beta_\ell),(\alpha_\ell,\beta)\in\mtop_{X_Q}$ for $\ell\geq 1$, then
$$t(\alpha)=s(\eta_1)=t(\gamma_1)=s(\eta_2)\cdots=s(\eta_\ell)=t(\gamma_\ell)=s(\beta).$$
Hence $t(\alpha)=s(\beta)$, and so $(\alpha,\beta)\in\mtop_{X_Q}.$
\end{proof}

This shows that the locality set of every quiver $Q$ is \mreg. The converse of Proposition~\mref{prop:quireg}, that is, every \mreg locality set $(X,\mtop)$ comes from the locality set of a quiver $Q$, will be proved in Theorem~\mref{thm:qxq}~(\mref{it:xqx}).

We next show that any locality set can be embedded into a \mreg quiver.

\begin{defn} \mlabel{de:regenv}
	Let $X=(X,\top)$ be a locality set. A {\bf \mreg hull} of $X$ is a \mreg locality set $X^\reg$ with a locality homomorphism $\phi:X\to X^\reg$ satisfying the following universal property: for every \mreg locality set $X'$ and locality set homomorphism $\psi:X\to X'$, there is a unique locality set homomorphism $\widetilde{\psi}:X^\reg \to X'$ such that $\widetilde{\psi}\circ \phi=\psi$.
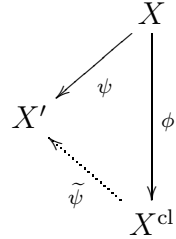
	$$ \xymatrix{ & X \ar[dd]^{\phi} \ar[dl]^\psi \\
		X'  \\
		& X^\reg \ar@{.>}[ul]^{\widetilde{\psi}} }
	$$
\end{defn}

Let $X=(X,\top_X)$ be a locality set. By construction of the locality set $X_{Q_X}=(X_{Q_X},\top_{X_{Q_X}})$, we have $X_{Q_X}=X$ as sets. Further, the locality relation $\top_{X_{Q_X}}$ is the alternating transitive closure of $\top$ and hence contains $\top$. Thus we have the natural injection of the locality sets
$$ \phi=\phi_X: (X,\top_X)\to (X_{Q_X},\top_{X_{Q_X}})$$
that is the identity of the sets and the inclusion on the relations.

\begin{prop}
	Let $X$ be a locality and let $X_{Q_X}$ be the locality set from the quiver $Q_X$ from Definition~\mref{defn:quiloc}.
	Then with the locality homomorphism $\phi_X:X\to X_{Q_X}$ defined above,  $X_{Q_X}$ is a \mreg hull of $X$.
\end{prop}

\begin{proof}
We have shown in Proposition~\mref{prop:quireg} that $X_{Q_X}$ is \mreg. So we just need to verify that $X_{Q_X}$, together with the natural inclusion $\phi_X$, satisfies the universal property in Definition~\mref{de:regenv}. Let $X'=(X',\top')$ and let $\psi:X\to X'$ be a locality map. Since $X_{Q_X}=X$ as a set, we can uniquely define $\tilde{\psi}: X^\reg \to X'$ simply to be $\phi$. Further, any pair $(\alpha,\beta)\in \top_{X_{Q_X}}$ comes from an alternating string of pairs in $\top_X$ of the form
$$ (\alpha,\beta_1),(\alpha_1,\beta_1), (\alpha_1,\beta_2), (\alpha_2,\beta_2), \cdots, (\alpha_n,\beta_n), (\alpha_n,\beta).$$
Since $\psi:X\to X'$ is a locality map, this string gives an alternating string of pairs in $\top'$ of the form
$$ (\alpha',\beta_1'),(\alpha_1',\beta_1'), (\alpha_1',\beta_2'), (\alpha_2',\beta_2'), \cdots, (\alpha_n',\beta_n'), (\alpha_n',\beta').$$
Here for clarity, we denote $\gamma'=\psi(\gamma)$ for all $\gamma\in \top_X$.
Since $X'$ is \mreg, we have
$(\psi(\alpha),\psi(\beta))=(\alpha',\beta')\in \top'$. Thus
$(\psi\times \psi)(\top_{X_{Q_X}}) \subseteq X'\times X'$ and hence $\tilde{\psi}$ is a locality homomorphism. By construction, this is also the unique locality homomorphism from $X^\reg$ to $X'$ extending $\psi$. This completes the verification of the desired universal property.
\end{proof}

The following result states that the path locality semigroup of a quiver $Q$ does not change when the quiver is replaced by the quiver of the locality set $X_Q$.
\begin{prop} Let $Q$ be a quiver and let $X_Q$ be the locality set of $Q$.
Let $Q_{X_Q} $ be the quiver of $X_Q$. Then $\calp_{Q_{X_Q}}=\calp_Q$.
	\mlabel{pp:pathsame}
\end{prop}

\begin{proof}
	For a quiver $Q=(Q_0,Q_1,s,t)$, the corresponding locality set $(X_Q,\top_{X_Q})$ is defined by $X_Q=Q_1$ and
	$$\top_{X_Q}=\{(\alpha,\beta)\in X_Q\times X_Q\,|\, t(\alpha)=s(\beta)\}.$$
	Further, for the quiver $Q_{X_Q}=(Q_0',Q_1',s',t')$ of $X_Q$, we have $Q_1'=X_Q=Q_1$, and by Proposition~\mref{prop:quireg}, $X_Q$ is a \full locality set. Then by Lemma~\mref{lem:locequiv}, $t'(\alpha)=s'(\beta)$ if and only if $(\alpha,\beta)\in \top_{X_Q}$. Hence $t(\alpha)=s(\beta)$ if and only if $t'(\alpha)=s'(\beta)$. Therefore the paths of $Q$ are the same as the paths of $Q_{X_Q}$.
\end{proof}

\section{One-to-one correspondence between \mreg locality sets and quivers}
\mlabel{ss:all}
From any quiver $Q$, we obtain the locality set $X_Q$ of the quiver $Q$ defined in Definition~\mref{defn:quiloc} which further gives the quiver $Q_{\mOmega_Q}$ from the locality set $\mOmega_Q$ defined in Definition~\mref{defn:locqui}. We will characterize the derived quiver $Q_{X_Q}$.

We use an example to illustrate the relationship between the quivers $Q$ and $Q_{X_Q}$.

\begin{exam} \mlabel{ex:qlq1}
	Let $Q=(Q_0, Q_1)$ be defined by $Q_0=\{x,y,z,w\}$  and $Q_1=\{\alpha: x\to z, \beta: y\to z\}$.
	$$ \xymatrix{x \ar[r]^{\alpha} & z &w\\
		&y \ar[u]^{\beta} & }
	$$
	The corresponding locality set is $(X_Q,\mtop_{X_Q})$ defined by $X_Q:=Q_1$ and $$\mtop_{X_Q}:=\{(\sigma,\tau)\in Q_1\times Q_1\,|\, t(\sigma)=s(\tau)\}=\emptyset. $$
	Thus for $Q':=Q_{\mOmega_Q}=(Q'_0,Q'_1,s',t')$, we have $Q'_1=X=Q_1$ and
	$$Q'_0:=(X_s\sqcup X_t)/\!\!\equiv\,=X_s\sqcup X_t$$
	since the locality relation $\mtop_{X_Q}$ is empty.
	Thus denoting $X_s=\{x,y\}$ and $X_t=\{z_1,z_2\}$, we have $Q'_0=\{z_1,z_2,x,y\}$ and $Q'_1=\{\alpha:x \to z_1, \beta: y\to z_2\}$.
	$$ \xymatrix{x \ar[r]^{\alpha} & z_1&z_2 \\
		&& \ar[u]^{\beta} y}$$	
\end{exam}

Intuitively, $Q'$ is obtained from $Q$ by removing the isolated vertex $w$ and decoupling the multi-targeted vertex $z$ into two vertices $z_1$ and $z_2$. We give a general notion for this type of quivers.

\begin{defn}
	A quiver $Q=(Q_0,Q_1,s,t)$ is called {\bf \full} if for every $x\in Q_0$, we have
	\begin{enumerate}
		\item $|s^{-1}(x)\cup t^{-1}(x)|\geq 1$. In other words, $Q_0=\mim (s)\cup \mim(t)$, that is, every vertex $x\in Q_0$ is either the source or target of some arrow;
		\mlabel{it:fq1}
		\item if $|s^{-1}(x)\cup t^{-1}(x)|\geq 2$, then
		$x\in \mim (s)\cap \mim(t)$.
		\mlabel{it:fq2}
	\end{enumerate}
	\mlabel{defn:fullq}
\end{defn}

Intuitively, a quiver is \full if every vertex is either the source or target of at least one arrow, and the vertex is the sources or targets of multiple arrows only if it is both a source and a target (for possibly different arrows). To put it in another way, a quiver $Q$ is \full if every vertex of $Q$ is connected to at least one arrow, and is connected to two or more arrows only if it is both the source and target of some arrows.

To understand the general situation better, we look at another example.

\begin{exam}
	Let $Q=(Q_0,Q_1)$ be the quiver defined by $Q_0=\{x,y,z,w\}$ and $Q_1=\{\alpha:x\to z, \beta:y\to z, \gamma:z\to w\}$
	$$ \xymatrix{x \ar[r]^{\alpha} & z \ar[r]^\gamma &w\\
		&y \ar[u]^{\beta} & }
	$$
This quiver is \full.
For its derived quiver, we have
	$$ \mtop_{X_Q}=\{(\alpha,\gamma), (\beta,\gamma)\}.$$
	So for $Q':=Q_{X_Q}=(Q'_0,Q'_1,s',t')$, we have
	$Q'_1=X_Q=Q_1=\{\alpha,\beta,\gamma\}$.
	$$ X_s\sqcup X_t=\{\alpha_s,\beta_s,\gamma_s,\alpha_t,\beta_t,\gamma_t\}$$
	on which the relation $\sim$ define from $\mtop_{X_Q}$ is
	$$\sim =\{(\alpha_t,\gamma_s),(\beta_t,\gamma_s)\}.$$
	Hence the equivalence classes of the equivalence relation $\equiv$ spanned by $\sim$ are the connected components of $X_s\sqcup X_t$ from $\top$:
	$$ Q'_0:=(X_s\sqcup X_t)/\!\!\equiv\, =\{ \overline{\alpha_s}, \overline{\beta_s}, \overline{\gamma_s}(=\overline{\alpha_t}=\overline{\beta_t}), \overline{\gamma_t}\}.$$
	Then together with $s', t'$, the quiver $Q'$ is depicted in the diagram:
	$$ \xymatrix{\overline{\alpha_s} \ar[r]^{\alpha} & \overline{\gamma_s} \ar[r]^\gamma &\overline{\gamma_t}\\
		&\overline{\beta_s} \ar[u]^{\beta} & }
	$$
	It is isomorphic to the original $Q$.
\end{exam}

We next show that any quiver can be lifted to a \full quiver.

\begin{defn} \mlabel{de:regcover}
	Let $Q$ be a quiver. A {\bf \full cover} of $Q$ is a \full quiver $Q^\ful$ with a quiver homomorphism $\phi:Q^\ful\to Q$ satisfying the following universal property: for every \full quiver $Q'$ and quiver homomorphism $\psi:Q'\to Q$, there is a unique quiver homomorphism $\widetilde{\psi}:Q' \to Q^\ful$ such that $\phi\circ \widetilde{\psi}=\psi$.
	$$ \xymatrix{ & Q^\ful \ar[dd]^{\phi} \\
		Q'\ar[ur]^{\widetilde{\psi}} \ar[dr]^{\psi} \\
		& Q  }
	$$
\end{defn}

We now show that the quiver $Q_{X_Q}$ is a \full cover of $Q$. To simplify the notations, denote $Q_{X_Q}=Q^*=(Q_0^*,Q_1^*,s^*,t^*)$. Then following Definitions~\mref{defn:quiloc} and \mref{defn:locqui}, we have
$$X:=X_Q=Q_1, \quad Q^*_1=X_Q=Q_1.$$
Also take $X_s=\{\omega_s\,|\,\omega\in Q_1\}$ and $X_t=\{\omega_t\,|\,\omega\in Q_1\}$ in bijection with $X$ and mutually disjoint. Define  $\equiv$ to be the equivalence relation on $X_s\sqcup X_t$ spanned by the relation
$$ \sim: =\{(u,v):=(\alpha_t,\beta_s)\in (X_s\sqcup X_t)\times (X_s\sqcup X_t)\,|\, t(\alpha)=s(\beta) \text{ for some }\alpha,\beta\in Q_1\}.$$
Then we have
$$ Q_0^*:= (X_s\sqcup X_t)/\!\!\equiv. $$

Now we define a quiver homomorphism $\phi=\phi_Q:=(\phi_0,\phi_1):Q^*\to Q$ by
$$ \phi_1: Q_1^*=Q_1\longrightarrow Q_1, \quad \alpha\mapsto \alpha, \tforall \alpha\in Q_1^*,$$
$$ \phi_0: Q_0^*=(X_s\sqcup X_t)/\!\!\equiv\, \longrightarrow Q_0, \quad \ol{u}\mapsto  \left\{\begin{array}{ll} s(\alpha),  & \text{if } u=\alpha_s\in X_s, \\
	t(\alpha), & \text{if } u=\alpha_t\in X_t \end{array} \right .
$$
Thus $\phi_1$ is simply the identity map. To show that $\phi_0$ is well defined, we first check that, if $\ol{u}, \ol{v}\in Q_0^*$ with $u\sim v$, then $\phi_0(\ol{u})=\phi_0(\ol{v})$. But $u\sim v$ means that $u=\alpha_t, v=\beta_s$ and $t(\alpha)=s(\beta)$. Hence $\phi_0(\ol{u})=t(\alpha)=s(\beta)=\phi_0(\ol{v})$. Since being equal is an equivalence relation and $\equiv$ is the equivalent relation spanned by $\sim$, we have $\phi_0(\ol{u})=\phi_0(\ol{v})$ if $u\equiv v$. Hence $\phi_0$ is well  defined.

\begin{prop}
	Let $Q$ be a quiver and let $Q_{X_Q}$ be the quiver from the locality set $X_Q$ from Definitions~\mref{defn:quiloc} and \mref{defn:locqui}.
	Then with the quiver homomorphism $\phi_Q:Q_{X_Q}\to Q$ defined above,  $Q_{X_Q}$ is a \full cover of $Q$.
\end{prop}

\begin{proof}
	Let $Q'=(Q_0',Q_1',s',t')$ be a \full quiver and let $\psi=(\psi_0,\psi_1): Q'\to Q$ be a quiver homomorphism. We define a quiver homomorphism
	$$\widetilde{\psi}=(\widetilde{\psi}_0,\widetilde{\psi}_1): Q'\to Q^*$$
	as follows. First noting that $Q_1^*=Q_1$, we can simply define
	$$\widetilde{\psi}_1=\psi_1:Q_1'\to Q_1^*.$$
	
	Next, since $Q'$ is \full, we have $Q_0'=\mim (s')\cup \mim (t')$. Thus we can define the correspondence
	$$ \widetilde{\psi}_0: Q'_0\to Q_0^*=(X_s\sqcup X_t)/\!\!\equiv , \quad u\mapsto \left\{\begin{array}{ll}
		\overline{{\psi_1(\sigma)}_s}, &\text{ if } u=s'(\sigma)\in \mim (s'), \sigma\in Q'_1, \\
		\overline{{\psi_1(\sigma)}_t}, & \text{ if } u=t'(\sigma)\in \mim(t'), \sigma\in Q'_1.
	\end{array} \right .
	$$
	We verify that the correspondence is well-defined and hence defines a function. First if $u=s'(\sigma)=s'(\tau)$ or $u=t'(\sigma)=t'(\tau)$ for distinct $\sigma, \tau\in Q'_1$, then by \fuly, $u$ is in $\mim(s')\cap \mim(t')$.
	Thus we only need to show the compatibility when $u$ is in the intersection $\mim(s')\cap \mim(t')$. The latter means $u=s'(\sigma)=t'(\tau)$. Then $s(\psi_1(\sigma))=\psi_1(s'(\sigma))=\psi_1(t'(\tau))=t(\psi_1(\tau))$. This means that $\psi_1(\tau)_t \sim\psi_1(\sigma)_s$ and hence $\overline{{\psi_1(\sigma)}_s}=\overline{{\psi_1(\tau)}_t}$, as needed. In summary, we obtain a quiver homomorphism $\widetilde{\psi}:Q'\to Q^*$. By the construction, we have $\phi\circ \widetilde{\psi}=\psi$.

Finally suppose that a quiver homomorphism $\tilde{\tilde{\psi}}:Q'\to Q^*$ also satisfies $\phi\circ \tilde{\tilde{\psi}}=\psi$. Then from $\phi_1:Q_1^*=Q_1\to Q_1$ is the identity we have
	$$  \tilde{\tilde{\psi}}_1=\phi_1\circ \tilde{\tilde{\psi}}_1 =\psi_1=\phi_1 \circ \tilde{\psi}_1=\tilde{\psi}_1.$$
	Also, we must have $\tilde{\tilde{\psi}}_1=\psi_1=\tilde{\psi}_1$. Consequently,
	$$ \tilde{\tilde{\psi}}_0(s'(\sigma)) =s^*(\tilde{\tilde{\psi}}_1(\sigma))
	= s^*(\tilde{\psi}_1(\sigma))
	= \tilde{\psi}_0(s'(\sigma)),$$
	and similarly,
	$\tilde{\tilde{\psi}}_0(t'(\sigma))
	= \tilde{\psi}_0(t'(\sigma))$.
	Therefore, $\tilde{\tilde{\psi}}_0=\tilde{\psi}_0$ and $\tilde{\tilde{\psi}}=\tilde{\psi}$, proving the uniqueness of $\tilde{\psi}$.
\end{proof}

Here is another way to characterize \full quivers.

\begin{prop}
	Let $Q=(Q_0,Q_1,s,t)$ be a quiver. Define sets
	$$Q_s:=\{\omega_s\,|\,\omega\in Q_1\}, \quad
	Q_t:=\{\omega_t\,|\,\omega\in Q_1\}$$
	to be disjoint sets and in bijection with $Q_1$ by
	$$\left\{\begin{array}{ll}
		Q_1\to Q_s, &\omega \mapsto \omega_s, \\
		Q_1\to Q_t, &\omega \mapsto \omega_t, \end{array} \right. \quad \omega\in Q_1.$$
	Define a map
	$$ \pi: Q_s\sqcup Q_t\to Q_0, \quad
	\omega_s\mapsto s(\omega),\quad
	\omega_t\mapsto t(\omega), \quad \omega\in Q_1.$$
	Then the quiver $Q$ is  \full if and only if
	\begin{enumerate}
		\item $\pi$ is onto, and
		\item the relation $\equiv$ defined by the fibers of $\pi$:
		$$ u \equiv v \Leftrightarrow \pi(u)=\pi(v) \quad \tforall u, v\in Q_s\sqcup Q_t,$$
		is the equivalence relation spanned by the relation
		$$\sim:=\{(\alpha_t,\beta_s)\in Q_t\times Q_s\subseteq (Q_s\sqcup Q_t)\times (Q_s\sqcup Q_t)\,|\,t(\alpha)=s(\beta)\}.$$
	\end{enumerate}
	\mlabel{prop:fullq}
\end{prop}
\begin{proof}
	($\Longrightarrow$) Since $Q_0=\mim (s)\cup \mim(t)$, for every $x\in Q_0$, there exists $\alpha\in Q_1$ such that $s(\alpha)=x$ or $t(\alpha)=x$. Then $\pi(\alpha_s)=s(\alpha)=x$ or $\pi(\alpha_t)=t(\alpha)=x$, and so $\pi$ is onto. Let $\equiv'$ be the equivalence relation spanned by the relation $\sim$ defined as above. Let $u,v\in Q_s\sqcup Q_t$. Then there are four cases: (1) $u,v\in Q_s$; (2) $u\in Q_s, v\in Q_t$; (3) $u\in Q_t, v\in Q_s$; (4) $u,v\in Q_t$. We just verify $u\equiv v\Leftrightarrow u\equiv' v$ for the first case $u=\alpha_s,v=\beta_s\in Q_s$. The other cases  are similar. By $u\equiv v$, we obtain $\pi(u)=\pi(v)$, and then $s(\alpha)=s(\beta)=:x\in Q_0$. By Definition~\mref{defn:fullq}~(\mref{it:fq2}), there is $\gamma\in Q_1$ such that $t(\gamma)=x$, and so $s(\alpha)=t(\gamma)=s(\beta)$. This gives $\gamma_t \sim \alpha_s$ and $\gamma_t\sim \beta_s$, and hence $\alpha_s\equiv'\beta_s$, proving $u\equiv'v$. On the other hand, let $u\equiv' v$. Then we obtain $\alpha_s\equiv' {\gamma_1}_t\equiv'{\eta_1}_s\equiv'\cdots {\gamma_k}_t\equiv'\beta_s$ with $({\gamma_1}_t,\alpha_s),({\gamma_1}_t,{\eta_1}_s),\cdots,({\gamma_k}_t,\beta_s)\in \sim$. Thus
	$s(\alpha)=t(\gamma_1)=\cdots=t(\gamma_k)=s(\beta)$, which implies $\pi(\alpha_s)=\pi(\beta_s)$ and thus $u\equiv v.$
	\noindent
	
	($\Longleftarrow$) Let $x\in Q_0$. The condition that $\pi$ is onto gives $\pi(\omega_s)=x$ or $\pi(\omega_t)=x$ for some $\omega\in Q_1$. Then we obtain $s(\omega)=x$ or $t(\omega)=x$, and so $Q_0=\mim(s)\cup \mim(t)$. Let $|s^{-1}(x)\cup t^{-1}(x)|\geq 2$. Then there exists $\alpha,\beta\in Q_1$ such that $s(\alpha)=x=t(\beta)$, or $t(\alpha)=x=s(\beta)$, or $s(\alpha)=x=s(\beta)$, or $t(\alpha)=x=t(\beta)$. The first two cases mean $x\in\mim(s)\cap \mim(t)$.  By the third case $s(\alpha)=s(\beta)$, we have $\pi(\alpha_s)=\pi(\beta_s)$, and so $\alpha_s\equiv' \beta_s$. Then there exists $\gamma\in Q_1$ such that $s(\alpha)=t(\gamma)=s(\beta)=x$, giving $x\in \mim(s)\cap \mim(t)$. The fourth case is similarly checked. Hence $Q$ is \mreg.
\end{proof}

The close relationship between locality sets and quivers can be summarized in the following theorem. It resembles a Galois connection yet to be explored.

\begin{theorem}
	Let $\calq$ denote the set of (finite) quivers and let $\calq^{\ful}\subseteq \calq$ be the subset of \full (finite) quivers.  Let $\call\cals$ denote the set of (finite)  locality sets and let $\call\cals^{\reg}\subseteq \call\cals$ denote the subset of \mreg (finite) locality sets.
	Define the maps
	\begin{equation}
		\Psi:\calq\to \call\cals, \quad Q\mapsto X_Q
		\mlabel{eq:qtol}
	\end{equation}
	and
	\begin{equation}
		\Phi:\call\cals\to \calq, \quad X\mapsto Q_X.
		\mlabel{eq:ltoq}
	\end{equation}	
	\begin{enumerate}
		\item  The quiver $Q_{X}$ of a locality set $X$ is \full. In particular, the quiver $Q_{\mOmega_Q}$ is \full;
		\mlabel{it:qx}
		\item  The composition $\Phi\Psi$ is a projection onto $\calq^{\ful}$. More precisely, for every $Q\in \calq^{\ful}$, $Q_{X_Q}=Q$, and further $\Phi(\Psi(\calq))=\calq^{\ful}$.
		\mlabel{it:psf}
		\item The composition $\Psi\Phi$ is the projection onto $\call\cals^{\sreg}$. More precisely, for every $(X,\mtop_X) \in \call\cals^{\sreg}$,  $(X_{Q_X},\mtop_{X_{Q_X}})=(X,\mtop_X)$, and further
		$\Psi(\Phi(\call\cals))=\call\cals^{\sreg}$.
		\mlabel{it:xqx}
	\end{enumerate}
	\mlabel{thm:qxq}
\end{theorem}
The theorem can also be recast in terms of category equivalence.

\begin{proof}
	(\mref{it:qx}) Let $X:=(X,\mtop_X)$ be in $\call\cals$. Let $\equiv$ denote the equivalent relation on $X_s\sqcup X_t$ spanned by the relation $\alpha_t\sim\beta_s \Leftrightarrow (\alpha,\beta)\in \mtop_X$, where $\alpha_t\in X_t$ and $\beta_s\in X_s$. By Definition~\mref{defn:quiloc}, we get
	$$Q_X=(Q_0,Q_1,s_Q,t_Q),$$
	where
	$$Q_1=X, \quad Q_0=(X_s\sqcup X_t)/\!\!\equiv,\quad  s_Q: Q_1\to Q_0, \alpha\mapsto \ol{\alpha_s}, \quad
	t_Q: Q_1\to Q_0, \alpha\mapsto \ol{\alpha_t}.$$
	We let $Q_s:=X_s$ and $Q_t:=X_t$. By the definition of $\pi$ in Proposition~\mref{prop:fullq}, we have
	$$\pi:Q_s\sqcup Q_t\to Q_0, \alpha_s\mapsto s_Q(\alpha), \alpha_t\mapsto t_Q(\alpha), \alpha_s\in Q_s,\,\alpha_t\in Q_t.$$
	Since $Q_0=(Q_s\sqcup Q_t)/\!\!\equiv, s_Q(\alpha)=\ol{\alpha_s}$ and $t_Q(\alpha)=\ol{\alpha_t}$, $\pi$ is onto.
	Let $\equiv'$ denote the equivalent relation on $Q_s\sqcup Q_t$ defined by $\alpha\equiv'\beta \Leftrightarrow \pi(\alpha)=\pi(\beta)$.
	We next verify $\equiv'=\equiv$, that is,
	\begin{equation}
		\alpha\equiv' \beta \Leftrightarrow\alpha\equiv\beta\quad \forall \alpha,\beta\in Q_s\sqcup Q_t.
		\mlabel{eq:qst}
	\end{equation}
	There are four cases to consider: (1) $\alpha,\beta\in Q_s$; (2) $\alpha\in Q_s,\beta\in Q_t$; (3) $\alpha\in Q_t, \beta\in Q_s$; (4) $ \alpha,\beta\in Q_t$. We just verify that Eq.~(\mref{eq:qst} holds for Case (2) since the other cases are similar.  So  let  $\alpha=\alpha_s\in Q_s$ and let $\beta=\beta_t\in Q_t$. Then
	\begin{eqnarray*}\alpha_s\equiv'\beta_t &\Leftrightarrow& \pi(\alpha_s)=\pi(\beta_t)\\
		&\Leftrightarrow& s_{Q}(\alpha)=t_{Q}(\beta)\\
		&\Leftrightarrow& \ol{\alpha_s}=\ol{\beta_t}\\
		&\Leftrightarrow& \alpha_s\equiv \beta_t,
	\end{eqnarray*}
	as needed.	
	\smallskip
	
	\noindent
	(\mref{it:psf}) Let $Q:=(Q_0,Q_1,s,t)\in \calq^{\sfull}$. Then there is a surjective map
	$$\phi:Q_s\sqcup Q_t\to Q_0, \omega_t\mapsto t(\omega), \omega_s\mapsto s(\omega),\quad \omega\in Q_1, $$
	and $\equiv$ defined by the fibers of $\phi$ equals the equivalent relation induced by $\alpha_t\sim \beta_s\Leftrightarrow t(\alpha)=s(\beta)$. Let $Q':=Q_{X_Q}$. Then by the definition of $Q_{X_Q}$, we have $Q'_1=X_Q= Q_1$, $Q'_0=({X_Q}_s\sqcup {X_Q}_t)/\!\!\equiv\,=(Q_s\sqcup Q_t)/\!\!\equiv$ by $Q_s={X_Q}_s$ and $Q_t={X_Q}_t$, and
	$$ s_{Q'}: Q'_1\to Q'_0, \alpha\mapsto \ol{\alpha_s}, \quad
	t_{Q'}: Q'_1\to \Q'_0, \alpha\mapsto \ol{\alpha_t}, \quad \alpha\in Q'_1.$$ Define a map
	$$f:Q_0\to Q'_0, \quad x\to \ol{\alpha},$$
	where $\phi(\alpha)=x$ with $\alpha\in Q_s\sqcup Q_t$. If $\phi(\omega)=x$ for any $\omega\in Q_s\sqcup Q_t$, then $\ol{\omega}=\ol{\alpha}$. Thus $f$ is well defined.
	
	We next verify that $f:Q_0\to Q'_0$  is a bijection and  $x\overset{\alpha}{\to} y \Leftrightarrow f(x)\overset{\alpha'}{\to} f(y)$ for all $x,y\in Q_0$, $\alpha'\in Q'_1, \alpha\in Q_1$. If $\ol{\alpha}=f(x)=f(y)=\ol{\beta}$, then $\alpha\equiv \beta$, and so  $x=\phi(\alpha)=\phi(\beta)=y$. Thus, $f$ is injective. Let $\ol{\alpha}\in Q'_0$ with $\alpha\in Q_s\sqcup Q_t$.  Then  $\phi(\alpha)=x$ for some $x\in Q_0$, Thus $f(x)=\ol{\alpha}$, and hence $f$ is surjective.
	
	Now let $x\overset{\alpha}{\to} y\in Q_1 $. Then $\alpha_s\in Q_s$ and $\alpha_t\in Q_t$. Since $\phi(\alpha_s)=s(\alpha)=x$ and $\phi(\alpha_t)=t(\alpha)=y$,  we have $f(x)=\ol{\alpha_s}$, and $f(y)=\ol{\alpha_t}$. For $\alpha\in Q'_1$, we have $s_{Q'}(\alpha)=\ol{\alpha_s}$ and $t_{Q'}(\alpha)=\ol{\alpha_t}$. Thus we get $f(x)\overset{\alpha}{\to} f(y)$.
	
	Let $f(x)\overset{\alpha}{\to} f(y)$. Then $s_{Q'}(\alpha)=\ol{\alpha_s}=f(x)$ and $t_{Q'}(\alpha)=\ol{\alpha_t}=f(y)$. This gives $\phi(\alpha_s)=x$ and $\phi(\alpha_t)=y$. Thus $x=\phi(\alpha_s)=s(\alpha)$  and $y=\phi(\alpha_t)=t(\alpha)$ and so $x\overset{\alpha}{\to} y$.
	This proves that $Q_{X_Q}=Q$, and then together with  (\mref{it:qx})  $\Phi(\Psi(\calq))=\calq^{\sfull}$ holds.
	\smallskip
	
	\noindent
	(\mref{it:xqx}) Let $X:=(X,\mtop_X)$ be a \mreg locality set. By Definition~\mref{defn:quiloc}, we get $Q_1=X$ and $Q_0=(X_s\sqcup X_t)/\!\!\equiv$. This gives $X_{Q_X}=X$, and
	\begin{equation}
		\mtop_{X_{Q_X}}=\{(\alpha,\beta)\in X\times X\,|\, t_Q(\alpha)=s_Q(\beta)\}.
		\mlabel{eq:defnxqx}
	\end{equation}
	Since $t_Q(\alpha)=\ol{\alpha_t}=\ol{\beta_s}=s_Q(\beta)$, we have $\alpha_t\equiv \beta_s$, and hence, $(\alpha,\beta)\in \mtop_X$ by Lemma~\mref{lem:locequiv}. Thus $\mtop_{X_{Q_X}}\subseteq \mtop_X$. On the other hand, if $(\alpha,\beta)\in\mtop_X$, then $\alpha_t\equiv \beta_s$. This leads to $t_Q(\alpha)=s_Q(\beta)$, and so $(\alpha,\beta)\in \mtop_{X_{Q_X}}$ by Eq.~(\mref{eq:defnxqx}). Then $\mtop_X\subseteq \mtop_{X_{Q_X}}$, and we get $\mtop_X= \mtop_{X_{Q_X}}$. So $(X_{Q_X},\mtop_{X_{Q_X}})=(X,\mtop_X)$.
	
By Proposition~\mref{prop:quireg}, for any quiver $Q$, the locality set $X_Q$ is \mreg. In particular, $X_{Q_X}$ is \mreg, proving
$\Psi(\Phi(\call\cals))=\call\cals^{\sreg}$.
\end{proof}

\section{Quiver semigroups as free \fine locality semigroups}
\mlabel{subsec:fineLSG}
We finally show that the path locality semigroup from a quiver is the free object in a suitable category.
For this purpose, we introduce a class of locality semigroups conditioned by path locality semigroups.

\begin{defn}
A locality semigroup $(S,\top,\mu_S)$ is called a {\bf \fine} if for any $a, b, c\in \top$,
	\begin{enumerate}
		\item when $(a,b)$ is in $\top$, then $(ab,c)\in T$ if and only if $(b,c)\in T$;
		\item when $(b,c)$ is in $\top$, then $(a,bc)\in T$ if and only if $(a,b)\in T$.
	\end{enumerate}
\end{defn}

Then from the definition of a path locality semigroup, we obtain

\begin{prop}
Let $(\calp,\mtop_\calp,\mu_\calp)$ be the path locality semigroup as defined in Proposition~\mref{prop:pathsg}. Then
	$(\calp,\mtop_\calp,\mu_\calp)$ is a \fine locality semigroup.
\mlabel{prop:pathresg}
\end{prop}

The collection of \fine locality semigroups forms a full subcategory of the category of locality semigroups and we can talk about its free objects.

\begin{defn}
Let $(X,\mtop_X)$ be a locality set. A {\bf free \fine locality semigroup} on $(X,\mtop_X)$ is a \fine locality semigroup $(F_L(X),\mtop_F)$ together with a locality  map $j_X:
(X,\mtop_X)\to (\FL(X),\mtop_F)$ such that, for any \fine locality semigroup $(S,\mtop_S)$ and
any locality map $f:(X,\mtop_X)\to (S, \mtop_S)$, there exists a unique
locality semigroup homomorphism $\free{f}: (\FL(X),\mtop_F)\to (S,\mtop_S)$
such that $f=\free{f}\circ j_X$, that is, the following diagram
$$\xymatrix{ (X,\mtop_X)  \ar[rr]^{j_X}\ar[drr]_{f} && (\FL(X),\mtop_F) \ar[d]^{\free{f}} \\
&& (S,\mtop_S)}
$$
commutes.
\mlabel{defn:freelocsg}
\end{defn}

Let $(\mOmega,\mtop_X)$ be  a locality set and let $\calp:=\calp_\mOmega:=\calp_{Q_\mOmega}$ be the path locality semigroup of the quiver $Q_\mOmega$ of $\mOmega$. Then $\mtop_\calp=\{(\alpha,\beta)\in \calp\times\calp\,|\,t_Q(\alpha)=s_Q(\beta)\}$. If $(\alpha,\beta)\in \mtop_X$, then $\overline{\alpha_t}=\overline{\beta_s}$ in
$Q_0:=( \mOmega_s\sqcup  \mOmega_t)/\!\!\equiv$. By the definitions of $t_Q$ and $s_Q$, we have $t_Q(\alpha)=\overline{\alpha_t}=\overline{\beta_s}=s_Q(\beta) $, and so $\mtop_X\subseteq \mtop_\calp$.
Define a set map
\begin{equation}
j_\mOmega: \mOmega\to \calp,\quad\omega\mapsto \omega,\quad \omega\in \mOmega.
\end{equation}
Then $(j_\mOmega\times j_\mOmega)(\alpha,\beta)=(\alpha,\beta)\in \mtop_\calp$ for all $(\alpha,\beta)\in \mtop_\mOmega$, and so $j_\mOmega$ is a locality map. We now give the universal property of path locality semigroups.

\begin{theorem}
\begin{enumerate}
\item 	
Let $X=(X,\mtop_X)$ be a \mreg locality set and let $Q_X$ be the quiver of $X$. The path locality semigroup $\calp_{Q_X}$ together with the locality map $j_X$ is the free \fine locality semigroup on  $X$.
\mlabel{it:freeloc1}
\item For every quiver $Q$,  the path locality semigroup $\calp_Q$ is the free \fine locality semigroup on the locality set $(X_Q, \mtop_{X_Q})$.
\mlabel{it:freeloc2}
\end{enumerate}
\mlabel{thm:freeloc1}
\end{theorem}

Before proving the theorem,  we first give two preliminary results.
\begin{lem}	\mlabel{lem:fs}
	Let $(S,\mtop_S)$ be a  \fine locality semigroup and let $n\geq 2$.  Let $(X,\mtop_X)$ be a locality set and let $(x_i,x_{i+1})\in \mtop_X$ for $i=1,\cdots,n-1$. If $f:(X,\mtop_X)\to (S,\mtop_S)$ is a locality map, then
	\begin{equation}
		(f(x_1)\cdots f(x_{n-1}),f(x_{n}))\in \mtop_S
		\mlabel{eq:fs1}
	\end{equation}
	and
	\begin{equation}
		(f(x_1),f(x_2)\cdots f(x_{n}))\in \mtop_S.
		\mlabel{eq:fs2}
	\end{equation}
\end{lem}

\begin{proof}
	We prove Eqs.~(\mref{eq:fs1}) and (\mref{eq:fs2}) by induction on $n\geq 2$. For $n=2$, we have $(f(x_1),f(x_2))\in \mtop_S$ by $f$ being a locality map.
	Assume that Eqs.~(\mref{eq:fs1}) and (\mref{eq:fs2}) have been proved for $n\leq k$. Consider $n=k+1$.  Since $f$ is a locality map and $(x_i,x_{i+1})\in \mtop_X$ for $i=1,\cdots,k$, $(f(x_{i}),f(x_{i+1}))\in \mtop_S$ for all $i$, and especially $(f(x_k),f(x_{k+1}))\in\mtop_S$. By the induction hypothesis, we get
	$$(f(x_1)\cdots f(x_{k-1}), f(x_{k}))\in \mtop_S,\quad
	(f(x_1),f(x_2)\cdots f(x_{k}))\in \mtop_S$$
	and $ (f(x_2)\cdots f(x_{k}),f(x_{k+1}))\in \mtop_S$.
	Since $(S,\mtop_S)$ is a \fine locality semigroup,  we have
	$$(f(x_1)\cdots f(x_{k-1})f(x_k),f(x_{k+1}))\in \mtop_S$$
	and
	$$(f(x_1),f(x_2)\cdots f(x_{k})f(x_{k+1}))\in \mtop_S.$$
	This completes the induction.
\end{proof}

\begin{lem}\mlabel{lem:fmul}
	Let $(S,\mtop_S)$ be a \fine locality semigroup and let $m,n\geq 1$.  Let $(X,\mtop_X)$ be a locality set. Let $(x_i,x_{i+1})$, $(y_j,y_{j+1})\in \mtop_X$ for $i=1,\cdots,m-1$ and $j=1,\cdots,n-1$.  If $f:(X,\mtop_X)\to (S,\mtop_S)$  is a locality map and $(x_m,y_1)\in \mtop_X$, then
	\begin{equation}\mlabel{eq:fmul}
		(f(x_1)\cdots f(x_m),f(y_1)\cdots f(y_n))\in \mtop_S.
	\end{equation}
\end{lem}
\begin{proof} We first verify that Eq.~(\mref{eq:fmul}) holds for the special cases when $m=1$ or $n=1$.  If $m=1$, then by the conditions that
	$(x_1,y_1), (y_1,y_2),\cdots,(y_{n-1},y_n)\in \mtop_X$
	and  Eq.~(\mref{eq:fs2}), we obtain $(f(x_1),f(y_1)\cdots f(y_n))\in \mtop_S$. Similarly, if $n=1$, then by
	$(x_1,x_2), \cdots,(x_{m-1},x_m),(x_m,y_1)\in \mtop_X$ and Eq.~(\mref{eq:fs1}), we get $(f(x_1)\cdots f(x_m),f(y_1))\in\mtop_S$.
	
	Next we consider $m,n\geq 2$.  By Eqs.~(\mref{eq:fs1}) and ~(\mref{eq:fs2}) again, we obtain
	\begin{equation}
		(f(x_1)\cdots f(x_{m-1}), f(x_m))\in \mtop_S
		\mlabel{eq:fxm}
	\end{equation}
	and
	\begin{equation}
		(f(y_1),f(y_2)\cdots f(y_n))\in \mtop_S.
		\mlabel{eq:fbeta}
	\end{equation}
	Since $f$ is a locality map, we have $(f(x_m),f(y_1))\in \mtop_S$. Then by $(S,\mtop_S)$ being a \fine locality semigroup and Eq.~(\mref{eq:fxm}), we obtain $(f(x_1)\cdots f(x_{m-1})f(x_m),f(y_1))\in \mtop_S,$
and so $(f(x_1)\cdots f(x_{m-1})f(x_m),f(y_1)f(y_2)\cdots f(y_n))\in \mtop_S$ by Eq.~(\mref{eq:fbeta}).
\end{proof}

\begin{proof} (of Theorem~\mref{thm:freeloc1})\\
\meqref{it:freeloc1} By Proposition~\mref{prop:pathresg}, the path locality semigroup $\calp:=\calp_{Q_X}$ is a \fine locality semigroup.
We next show that $(\calp,\mtop_\calp, \mu_\calp)$ satisfies the required universal property.
Let $(S,\mtop_S)$ be a \fine locality semigroup and let $f: (\mOmega,\mtop_\mOmega) \to (S,\mtop_S)$ be a locality map. Define a set map
$$
\bar{f}: \calp \to S,\quad
p\mapsto  \bar{f}(p):=f(\alpha_1)f(\alpha_2)\cdots f(\alpha_k),
$$
where $ p=\alpha_1\cdots\alpha_k\in \calp$,  and $\alpha_i\in \mOmega$ for $1\leq i\leq k.$  Since $(\alpha_i,\alpha_{i+1})\in\mtop_\mOmega$ for $i=1,\cdots,k-1$, we have $(f(\alpha_i),f(\alpha_{i+1}))\in\mtop_S$ and thus $\bar{f}(p)$ is well-defined by Lemma~\mref{lem:fs}.
By the definition of $\bar{f}$, $f=\bar{f}\circ j_\mOmega$.
We next prove that $\bar{f}$ is  a  locality semigroup homomorphism, that is, $(\bar{f}\times\bar{f})(p_1,p_2)=(\bar{f}(p_1),\bar{f}(p_2))\in \mtop_S$  and $\bar{f}(p_1p_2)=\bar{f}(p_1)\bar{f}(p_2)$ for all $(p_1,p_2)\in \mtop_\calp$. Let $(p_1,p_2)\in \mtop_\calp$.
Suppose that $p_1=\alpha_1\cdots\alpha_m$ and $p_2=\beta_1\cdots \beta_n$ with $m,n\geq 1$, where $\alpha_i, \beta_j\in \mOmega$ for $1\leq i\leq m $ and $1\leq j\leq n$. Then we obtain $(\alpha_i,\alpha_{i+1})\in \mtop_\mOmega$ and  $(\beta_j,\beta_{j+1})\in \mtop_\mOmega$ for $i=1,\cdots, m-1$ and $j=1,\cdots,n-1$. By $(p_1,p_2)\in \mtop_\calp$,
$$t_Q(\alpha_m)=t_Q(p_1)=s_Q(p_2)=s_Q(\beta_1),$$
and so $(\alpha_m,\beta_1)\in \mtop_X$ by Lemma~\mref{lem:locequiv}.
Thus, by Lemma~\mref{lem:fmul}, we get
$$(f(\alpha_1)\cdots f(\alpha_{m-1})f(\alpha_m),f(\beta_1)f(\beta_2)\cdots f(\beta_n))\in \mtop_S.$$
This means that $(\bar{f}(p_1),\bar{f}(p_2))$ is in $\mtop_S$. Now we prove $\bar{f}(p_1p_2)=\bar{f}(p_1)\bar{f}(p_2)$.
By the definition of $\bar{f}$, we have
\begin{eqnarray*}
\bar{f}(p_1p_2)&=&\bar{f}(\alpha_1\cdots\alpha_m\beta_1\cdots\beta_n)\\
&=&f(\alpha_1)\cdots f(\alpha_m)f(\beta_1)\cdots f(\beta_n)\\
&=&\bar{f}(p_1)\bar{f}(p_2).
\end{eqnarray*}
To complete the proof, we finally verify the uniqueness of $\bar{f}$. Assume that there is another  locality semigroup homomorphism $\tilde{f}:(\calp,\mtop_\calp)\to (S,\mtop_S)$ such that $f=\tilde{f}\circ j_\mOmega$. For every $p=\alpha_1\cdots\alpha_k\in \calp$, we have
\begin{eqnarray*}
\tilde{f}(p)&=&\tilde{f}(\alpha_1\cdots \alpha_k)\\
&=&\tilde{f}(\alpha_1)\cdots\tilde{f}(\alpha_k)\\
&=&\tilde{f}(j_\mOmega(\alpha_1))\cdots\tilde{f}(j_\mOmega(\alpha_k))\\
&=&f(\alpha_1)\cdots f(\alpha_k)\\
&=&\bar{f}(p).
\end{eqnarray*}
Thus $\tilde{f}=\bar{f}$, as desired.

\smallskip

\noindent
\meqref{it:freeloc2}
By Proposition~\mref{pp:pathsame}, the path locality semigroup $\calp_Q$ is identified with the path locality semigroup $\calp_{Q_{X_Q}}$. Thus by Item~\meqref{it:freeloc1}, $\calp_Q$ is the free \fine locality semigroup on the locality set $(X_Q,\top_{X_Q})$.
\end{proof}

\noindent
{\bf Acknowledgements.} This work was supported by the National Natural Science Foundation of
China (Grant No.~11601199 and 11961031). Shanghua Zheng thanks Rutgers University-Newark for hospitality.

\end{document}